\documentclass[a4paper,12pt]{amsart}
\usepackage{geometry}
\usepackage{latexsym,amssymb,amsmath,graphicx,hyperref}
\usepackage{enumerate}
\numberwithin{equation}{section}
\newtheorem{theorem}{Theorem}[section]
\newtheorem{lemma}[theorem]{Lemma}
\newtheorem{proposition}[theorem]{Proposition}
\newtheorem{corollary}[theorem]{Corollary}

\theoremstyle{definition}

\newtheorem{example}[theorem]{Example}

\newcommand{\K}[0]{\mathbb{K}}

\DeclareMathOperator{\supp}{supp}

\DeclareMathOperator{\reg}{reg}

\DeclareMathOperator{\tor}{Tor}

\title{ The graded Betti numbers of truncation of ideals in polynomial rings}
\date{
}

\author{Chwas Ahmed, Ralf Fr{\"o}berg and Mohammed Rafiq Namiq}
\address{Chwas Ahmed, Department of Mathematics, College of Science, University of Sulaimani, Kurdistan Region, Iraq.}
\email{chwas.ahmed@univsul.edu.iq}

\address{Ralf Fr{\"o}berg, Department of mathematics, Stockholm university, Sweden}
\email{frobergralf@gmail.com}

\address{Mohammed Rafiq Namiq, Department of Mathematics, College of Science, University of Sulaimani, Kurdistan Region, Iraq.}
\email{mohammed.namiq@univsul.edu.iq}

\makeatletter
\@namedef{subjclassname@2020}{%
	\textup{2020} Mathematics Subject Classification}
\makeatother
\subjclass[2020]{Primary 13C40,	13D02; Secondary 13A02, 13C70.}

\keywords{Truncation, Betti numbers, index, linear resolution, polarization, Hilbert series}

\begin{document}
	
\begingroup
\def\uppercasenonmath#1{} 
\let\MakeUppercase\relax 
\maketitle
\endgroup

\begin{abstract}
	Let $R=\K[x_1,\dots,x_n]$, a graded algebra $S=R/I$ satisfies $N_{k,p}$ if $I$ is generated in degree $k$, and the graded minimal resolution is linear the first $p$ steps, and the $k$-index of $S$ is the largest $p$ such that $S$ satisfies $N_{k,p}$. Eisenbud and Goto have shown that for any graded ring $R/I$, then $R/I_{\geq k}$, where $I_{\geq k}=I\cap M^k$ and $M=(x_1,\dots,x_n)$, has a $k$-linear resolution (satisfies $N_{k,p}$ for all $p$) if $k\gg0$. 
	For a squarefree monomial ideal $I$, we are here interested in the ideal $I_k$ which is the squarefree part of $I_{\geq k}$. The ideal $I$ is, via 
	Stanley-Reisner correspondence, associated to a simplicial complex $\Delta_I$.
	In this case, all Betti numbers of $R/I_k$ for $k>\min\{\deg(u)\mid u\in I\}$, which of course is a much finer invariant than the index, can be determined from the Betti diagram of $R/I$ and the $f$-vector of $\Delta_I$. We compare our results with the corresponding statements for $I_{\ge k}$. (Here
	$I$ is an arbitrary graded ideal.) In this case we show that the Betti numbers of $R/I_{\ge k}$ can be determined from the Betti numbers of $R/I$
	and the Hilbert series of $R/I_{\ge k}$.
\end{abstract}


\section{Introduction}

Let $R=\K[x_1,\dots,x_n]$  and $I$ be an ideal of $R$. Denote the unique graded maximal ideal of $R$ by $M$. 
Green-Lazarsfeld introduced the condition $N_p$ which means that the resolution of $R/I$ is 2-linear the $p$ first times. Thus $N_1$ means that the ideal is generated in degree 2, $N_2$ means that furthermore all syzygies between the generators are linear. The largest $p$ for which the graded algebra $R/I$ satisfies $N_p$ is called the \emph{index} of $R/I$ (see \cite{GreenLazarsfeld1984,GreenLazarsfeld1985}).  Later the condition $N_{k,p}$ was introduced. This means that the resolution is $k$-linear the $p$ first times. Thus $R/I$ satisfies $N_{k,1}$ if $I$ is generated in degree $k$, it satisfies $N_{k,2}$ if furthermore the syzygies of the generators of $I$ are linear and so on. We call the largest $p$ for which $R/I$ satisfies $N_{k,p}$ the $k$-\emph{index} of $R/I$. That the $k$-index of $R$ is $\infty$ means that $R/I$ has a $k$-linear resolution. A quotient ring $S=R/I$ has a linear resolution if all syzygies are linear. That is, the resolution of $S$ is $k$-linear if $\tor^R_{i,j}(S,\K)=0$ if $j\neq k+i-1$ if $i>0$. 
For a positive integer $k$, Eisenbud and Goto in \cite{EisenbudGoto1984} define $I_{\geq k}=I\cap M^k$. They show that for any graded ring $R/I$, then $R/I_{\geq k}$ has a $k$-linear resolution for $k\geq\reg(I)$ (see {\cite[Theorem 1.2]{EisenbudGoto1984}). Moreover, a formula for the $k$-index is given by Eisenbud, Huneke, and Ulrich in {\cite[Proposition 1.6]{EisenbudHunekeUlrich2006}}. 
For a monomial ideal $I$, let $I_k$ be the squarefree part of $I_{\geq k}$. That is $I_k$ is the squarefree truncation of $I$ past $k$. If $I\subseteq R$ is a squarefree monomial ideal and $M_k=\left(u\mid u\text{ is a squarefree monomial in }M^k\right)$, then $I_k=I\cap M_k$.
In Theorem \ref{betti I_k}, we compute the graded Betti numbers of $R/I_k$ for $k>\min\{\deg(u)\mid u\in I\}$. The Betti numbers provide a deeper insight about the quotient ring than $k$-index invariant. For a monomial ideal (not squarefree) $I$, in Corollary \ref{I_k, I_geq k} we show that the truncation of $I$ past $k$ has a linear resolution if and only if the squarefree truncation of the polarization of $I$ past $k$, $\mathcal{P}(I)_{k}$, has a linear resolution.
Finally, for any graded ideal $I$, in Theorem \ref{betti I_geq k} we compute the graded Betti numbers of $R/I_{\geq k}$ for $k>\min\{\deg(u)\mid u\in I\}$.

\section{The graded Betti numbers of $I_k$}
In this section, let $I$ be a squarefree monomial ideal of $R$.
Define $J_k=\{ u\in M_k\mid u\notin I_k\}$. Let $e_u=\supp(u)$ where $\supp(u)=\{x_i\mid u=x_1^{a_1}\dots,x_n^{a_n}, a_i\neq 0\}$ and $G(I)$ be the set of minimal generators of $I$.
The Stanley-Reisner complex of $I$ is the simplicial complex consisting of faces which correspond to the squarefree monomials not in $I$, that is
$$\Delta_{I}=\{e_u\subset \{x_1,\dots,x_n\}\mid u\text{ is a squarefree monomial not in } I\}.$$
An element $F$ of $\Delta_{I}$ is a face of $\Delta_{I}$ of dimension $|F|-1$. A maximal face (with respect to inclusion) of $\Delta_{I}$ is a facet. Denote the set of the facets of $\Delta_{I}$ by $\mathcal{F}(\Delta_{I})$. The dimension of $\Delta_{I}$ is $\dim\Delta_{I}=\max\{|F|\mid F\in\mathcal{F}(\Delta_{I})\}-1$. The pure $k$-skeleton of $\Delta_{I}$ is a simplicial complex $\Delta_{I}^{[k]}$  whose facets are in $\mathcal{F}(\Delta_{I})$ with $|F|=k+1$.

Note that $I_{k+1}=\left(ux_i\mid\deg(u)=k,u\in I_k, x_i\notin e_u\right)+\left(v\mid \deg(v)>k,v\in I_k\right).$

\begin{lemma}\label{J_{k+1}}
	Let $I$ be a squarefree monomial  ideal of $R$.
	Then for $k\geq\min\{\deg(u)\mid u\in I\}$ we have
	$u\in J_{k+1}$ if and only if $e_u\in\Delta_{I_k}^{[k]}$.
\end{lemma}
\begin{proof}
	($\Rightarrow$) Let $u\in J_{k+1}$. Assume that $e_{u}\notin\Delta_{I_k}^{[k]}$. Then $u\in I_k$ and $u$ is a squarefree monomial of degree $k+1$. Thus $u\in I_{k+1}$ which is a contradiction.
	
	($\Leftarrow$) We have $u\in G(M_{k+1})$ since $e_u\in \Delta^{[k]}_{I_k}$. We claim $u\notin I_{k+1}$. Suppose to the contrary that $u\in I_{k+1}$. Then there exists $x_t\in e_u$ such that $\hat{u}_t\in I_k$. Hence $\hat{u}_t\notin J_k$ for some $x_t\in e_u$, a contradiction. The claim follows.
\end{proof}

In the following Proposition we find the  facets, i.e. maximal faces, of $\Delta_{I}$, then later use these facets to compute the Betti numbers of $I_k$.

\begin{proposition}\label{simplicial D_{I_k+1}}
	Let $I$ be a squarefree monomial ideal of $R$. Then for $k\geq\min\{\deg(u)\mid u\in I\}$ we have
	$$\mathcal{F}(\Delta_{I_{k+1}})=\{e_u\mid u\in G(I_k),\deg(u)=k\}\cup\{e_v\mid e_v\in\mathcal{F}(\Delta_{I_k}),\deg(v)\geq k\}.$$ 
\end{proposition}
\begin{proof}
	Let $P_k:=\{e_u\mid u\in G(I_k),\deg(u)=k\}$ and $Q_k:=\{e_v\mid e_v\in\mathcal{F}(\Delta_{I_k}),\deg(v)\geq k\}$.
	Let $F\in\mathcal{F}(\Delta_{I_{k+1}})$. Then $F=e_u$ for some $u\notin I_{k+1}$. We need to consider the following three cases.
	\begin{itemize}
		\item The case when $|F|>k$. If $e_u$ is not a face of $\Delta_{I_k}$, then $u\in I_k$. But $\deg(u)\geq k+1$ so $u\in I_{k+1}$, and this contradicts $e_u\in\Delta_{I_{k+1}}$. 
		It remains to show that $F$ is maximal in $\Delta_{I_{k}}$. Suppose $e_{ux}\in\Delta_{I_{k}}$ for some $x\notin e_u$. Note that $ux\not\in I_{k+1}$, otherwise having   $\deg(ux)>k+1$ implies $ux\in I_{k}$.  Hence $e_{ux}\in\Delta_{I_{k+1}}$, a contradiction. Therefore $F\in Q_k$.

		\item The case when $|F|=k$. If $u\notin I_k$, then $u\in J_k$. To obtain a contradiction assume that $e_{ux}\in\Delta_{I_k}$ for some $x\notin e_u$. Then $\deg({ux})=k+1$,  and Lemma~\ref{J_{k+1}} implies $ux\in J_{k+1}$. Hence  $e_{ux}\in\Delta_{I_{k+1}}$, again this contradicts the assumption $e_u\in\mathcal{F}(\Delta_{I_{k+1}})$. Therefore $e_u\in P_k$ or $e_u\in Q_k$.
		
		\item The case when $|F|<k$. We have $e_{ux}\notin\Delta_{I_{k+1}}$. Hence $ux\in I_{k+1}$ with $\deg({ux})\leq k$ which is a contradiction.
	\end{itemize}
	
	Conversely, let $e_u\in P_k\cup Q_k$, we show that $e_u\in\mathcal{F}(\Delta_{I_{k+1}})$ by considering the following three cases:
	\begin{itemize}
		\item
	 When $|e_u|=k$. If $e_u\in P_k$, then $ux\in I_{k+1}$ for all $x\notin e_u$. It follows that $e_u\in\mathcal{F}(\Delta_{I_{k+1}})$.
	 Now if $e_u\in Q_k$, then $u\not\in I_k$. Suppose that $ux\not\in I_{k+1}$ for some $x\notin e_u$. Then $ux\notin I_k$ otherwise $ux\in I_{k+1}$. Hence $e_{ux}\in\Delta_{I_k}$ which is a contradiction. Therefore $ux\in I_{k+1}$ for all $x\notin e_u$. Hence $e_u\in\mathcal{F}(\Delta_{I_{k+1}})$. 	
	\item If $|e_u|=k+1$, then $e_{u}\in\Delta_{I_k}$. By Lemma \ref{J_{k+1}} we have $u\in J_{k+1}$, hence $e_u\in\Delta_{I_{k+1}}$. If $e_u$ is not a facet of $\Delta_{I_{k+1}}$, then there an $x\notin e_u$ such that $e_{ux}\in\Delta_{I_{k+1}}$. Then $ux\notin I_{k+1}$, and so $ux\notin I_{k}$. Hence $e_{ux}\in\Delta_{I_{k}}$ some $x\notin e_u$ which is  a contradiction.
\item	When $|e_u|\geq k+2$. If $u\in I_{k+1}$ then $u\in I_k$, and so $e_u\notin\Delta_{I_{k}}$, which is also a contradiction. The face $e_u$ is maximal in $\Delta_{I_{k+1}}$ by the same argument used in the case $|e_u|=k+1$. 
	\end{itemize}
	
\end{proof}

For a fixed $t$, we shall write $(f_{-1}^t,f_{0}^t,\dots, f_{d-1}^t)$ for the $f$-vector of $\Delta_{I_t}$.

\begin{corollary}\label{f-vector I_k}
	Let $I$ be a squarefree monomial ideal of $R$ and $f(\Delta_{I})=(f_{-1},f_0,f_1,\dots,f_{d-1})$. Then for $k\geq\min\{|e_u|\mid u\in I\}$ we have $$f(\Delta_{I_{k+1}})=\begin{cases}
		\left(\binom{n}{0},\binom{n}{1},\binom{n}{2},\dots,\binom{n}{k},f_{k},\dots,f_{d-1}\right)\quad&\text{ if }k\leq d,\\ \left(\binom{n}{0},\binom{n}{1},\binom{n}{2},\dots,\binom{n}{k}\right)\quad&\text{ if }d<k<n.
	\end{cases}$$
\end{corollary}
\begin{proof}
	By Proposition \ref{simplicial D_{I_k+1}}, we have
	$f^{k+1}_{k-1}=f^{k}_{k-1}+|\{e_u\mid u\in G(I_k),\deg(u)=k\}|=|\{e_u\mid u\in J_k\}|+|\{e_u\mid u\in G(I_k),\deg(u)=k\}|=\binom{n}{k}$. Then the result follows by induction.
\end{proof}

The following Lemma of Hochster provides a very useful explanation of the graded Betti numbers of a Stanley-Reisner ring (see {\cite[Theorem 5.1]{Hochster1977}} or {\cite[Lemma 9]{Froberg2021}}). Let $I$ be a squarefree monomial ideal of $R=\K[x_1,\dots,x_n]$. We write
 $\Delta_W$ for the simplicial complex on the vertex set $W\subseteq\{x_1,\dots,x_n\}$, whose faces are $F\in\Delta_{I}$ with $F\subseteq W$.

\begin{lemma}[\textbf{Hochster}]\label{hochster}
Let $W\subseteq\{x_1,\dots,x_n\}$ and $K_{R(W)}$ be the part of the Koszul complex $K_R$ which is of degree $\delta(W)= (d_1,\dots,d_n)$, where $d_i=1$ if $x_i\in W$ and $d_i=0$ otherwise. Then $H_{i,\delta(W)}=H_i(K_{R(W)})\cong \tilde{H}_{|W|-i-1}\left(\Delta_W;\K\right)$.
\end{lemma}

Now we come to our main result. We have a precise formula for the graded Betti numbers of the graded algebra $R/I_k$ for $k>\min\{\deg(u)\mid u\in I\}$.

\begin{theorem}\label{betti I_k}
	Let $I$ be a squarefree monomial ideal of $R$. Then for $k>\min\{|e_u|\mid u\in I\}$ we have
	$$\beta_{i,i+j}(R/I_k)=\begin{cases}
		1\quad&\mbox{ if } i=0,j=0,\\
		0\quad&\mbox{ if }i\neq 0, 0\leq j\leq k-2,\\
		\alpha_{i,i+k-1}\quad&\mbox{ if }i\neq 0, j=k-1,\\
		\beta_{i,i+j}(R/I)\quad&\mbox{ if }i\neq 0, j\geq k.
	\end{cases}$$
	where \begin{equation}\label{formula of I_k}
		\displaystyle\alpha_{i,i+j}=\sum_{r=0}^{i+j}(-1)^{j-r}\binom{n-r}{i+j-r}f_{r-1}^k+\sum_{\scriptstyle \ell+m=i+j\atop\scriptstyle\ell< i}(-1)^{\ell-i-1}\beta_{\ell,\ell+m}(R/I).
	\end{equation}
\end{theorem}
\begin{proof}
	Let the $f$-vector of $\Delta_{I}$ be $(f_{-1},f_{0},f_{1}\dots, f_{d-1})$. Then Corollary \ref{f-vector I_k} implies that for $0\leq i\leq k-1$, $f^k_{i-1}=\binom{n}{i}$. Then  $\beta_{i,i+j}=0$ for all $0\leq j\leq k-2$ except $\beta_{0,0}=1$. Hence by Hochster's formula, Lemma \ref{hochster}, the only graded Betti numbers that change are
	$$\beta_{i,i+k-1}(R/I_k)=\sum_{\scriptstyle |W|=k-1,\atop W\subseteq \{x_1,\dots,x_n\}}\dim_{\K}\tilde{H}_{k-i-2}\left(\Delta_W;\K\right)=\alpha_{i,i+k-1}.$$
	where $\Delta_W$ is the simplicial complex on the vertex set $W$, whose faces are $F\in\Delta_{I_k}$ with $F\subseteq W$. Thus $\beta_{i,i+j}(R/I_{k})=\beta_{i,i+j}(R/I)$ for $j\geq k$. The Hilbert series of $\K[\Delta_{I}]$ is
	$$H_{\K[\Delta_{I}]}(t)=\sum_{r=0}^{d} \frac{f_{r-1} t^{r}}{(1-t)^{r}}=\frac{\sum_i(-1)^i\sum_j\beta_{i,i+j}t^{i+j}}{(1-t)^n}.$$
	We have that
	$$\displaystyle\sum_{r=0}^{d} \frac{f_{r-1} t^{r}}{(1-t)^{r}}\times(1-t)^n=\sum_{s=0}^{n}\sum_{r=0}^{s}(-1)^{s-r}\binom{n-r}{s-r}f_{r-1}t^s.$$
	Then from the Hilbert series of $\K[\Delta_{I}]$ we have 
	$$\sum_{s=0}^{n}\sum_{r=0}^{s}(-1)^{s-r}\binom{n-r}{s-r}f_{r-1}t^s=\sum_{i}(-1)^i\sum_j\beta_{i,i+j}t^{i+
		j}$$
	If $s=i+j$, then
	$$\sum_{r=0}^{s}(-1)^{s-r}\binom{n-r}{s-r}f_{r-1}=\sum_{i}(-1)^i\sum_{i+j=s}\beta_{i,s},\quad s=0,1,\dots,n.$$
	Let $j=k-1$. Since $\beta_{i,i+j}(R/I_k)\neq0$ whenever $j\geq k-1$, we have
	$$\beta_{i,i+j}(R/I_{k})=\sum_{r=0}^{i+j}(-1)^{j-r}\binom{n-r}{i+j-r}f_{r-1}^k+\sum_{\scriptstyle \ell+m=s\atop\scriptstyle \ell<i}(-1)^{\ell-i+1}\beta_{\ell,\ell+m}(R/I).$$
\end{proof}

\begin{corollary}
	Let $I$ be a squarefree monomial ideal of $R$. If $I$ has a linear resolution, then $I_k$ has a  linear resolution for any positive integer $k$.
\end{corollary}
\begin{proof}
	Let $I$ be generated in degree $d$. If $k\leq d$, then $I_k=I_d=I$. For the case $k>d$, one can apply Theorem \ref{betti I_k}.
\end{proof}

\begin{corollary}\label{reg I_k}
	Let $I$ be a squarefree monomial ideal of $R$ and $\reg(I)=d$. Then for any positive integer $k$ we have
	$$\reg(I_k)=\begin{cases}
		d\quad\mbox{ if }d\geq k,\\
		k\quad\mbox{ otherwise}.
	\end{cases}$$
\end{corollary}
\begin{proof}
	We have $\beta_{i,i+d}(I)\neq0$ for some $i$. Theorem \ref{betti I_k} implies that $\reg(I_k)=d$ whenever $d\geq k$ and $\reg(I_k)=k$ otherwise.
\end{proof}

\begin{corollary}\label{reg linear}
	Let $I$ be a squarefree monomial ideal of $R$. Then $\reg(I)=d$ if and only if $d$ is the smallest integer such that $d\geq\min\{\deg(u)\mid u\in I\}$ and $I_{d}$ has a linear resolution. In particular, $I_k$ has a linear resolution for all $k\geq\reg(I)$.
\end{corollary}
\begin{proof}
	We have $\beta_{i,i+d}(I)\neq0$ for some $i$. Hence Theorem \ref{betti I_k} implies that $I_d$ has a $d$-linear resolution but $I_{d-1}$ does not have a linear resolution.
	
	Conversely, there is a non-negative integer $m$ such that $I=I_{d-m}$. By Theorem \ref{betti I_k} we have $\beta_{i,i+d}(I)=\beta_{i,i+d}(I_{d-m})=\beta_{i,i+d}(I_d)$. Thus $\reg I=d$.
\end{proof}

\begin{example}\label{example 1}
	Let $I=\left(x_1x_2x_3,x_4x_5x_6x_7,x_1x_2x_4x_5x_8x_9\right)$ be an ideal of $R=\K[x_1,\dots,x_9]$. Then the $f$-vector of $\Delta_{I}$ is $(1,9,36,83,119,106,53,10)$ and the graded Betti numbers of $R/I$ are $\beta_{0,0}=1$, $\beta_{1,3}=1$, $\beta_{1,4}=1$, $\beta_{1,6}=1$, $\beta_{2,7}=2$, $\beta_{2,8}=1$ and $\beta_{3,9}=1$.
	Corollary \ref{f-vector I_k} implies that the $f$-vector of $\Delta_{I_{5}}$ is $(1,9,36,84,126,106,53,10)$.
	By Theorem \ref{betti I_k}, the graded Betti numbers of $T=R/I_5$ for $j=4$ are determined by
	$$\beta_{i,i+4}(T)=\sum_{r=0}^{4+i}(-1)^{4-r}\binom{9-4}{4+i-r}f_{r-1}^5+\sum_{\scriptstyle \ell+m=i+4\atop\scriptstyle \ell<i}(-1)^{\ell-i+1}\beta_{\ell,\ell+m}(R/I),$$
	and
	$\beta_{1,6}(T)=1$, $\beta_{2,7}(T)=2$,  $\beta_{2,8}(T)=1$ and $\beta_{3,9}(T)=1$.
	Then $\beta_{1,5}(T)=20$, $\beta_{2,6}(T)=49+\beta_{1,6}=50$, $\beta_{3,7}(T)=53+\beta_{2,7}=55$, $\beta_{4,8}(T)=30-\beta_{2,8}=29$ and $\beta_{5,9}(T)=7-\beta_{3,9}=6$.
	By Corollary \ref{reg linear}, $I_7$ has a linear resolution but $I_6$ does not have since $\reg(I)=7$.
\end{example}

\section{Polarizations}
Let $I=(u_1,\dots,u_t)$ be a monomial ideal of $R=\K[x_1,\dots,x_n]$ with $u_j=\prod_{i=1}^n x_i^{a_{i,j}} $ for $1\leq j\leq t$. For ${1\leq i\leq n}$, let $a_i=\max\{a_{i,j}\mid 1\leq j\leq t\}$. The polarization of $I$, denoted by $\mathcal{P}(I)$, is a squarefree monomial ideal of a polynomial ring $$S=\K[x_{1,1},x_{1,2},\dots,x_{1,a_{1}},x_{2,1},x_{2,2},\dots,x_{2,a_{2}},\dots,x_{n,1},x_{n,2},\dots,x_{n,a_{n}}]$$
with minimal generators $\mathcal{P}(u_1),\dots,\mathcal{P}(u_t)$ where $$\mathcal{P}(u_j)=\prod_{i=1}^n\prod_{l=1}^{a_{i,j}} x_{i,l}, 1\leq j\leq t.$$
In 1982, Fr\"oberg proved the following Lemma.
\begin{lemma}[\cite{Froberg1982}]\label{Froberg}
	Let $I=(u_1,\dots,u_t)$ be a monomial ideal of $R$ such that for each $i$, the variable $x_{i,a_{i}}$ appears in at least on of the monomials $\mathcal{P}(u_1),\dots,\mathcal{P}(u_t)$. Then for $1\leq i\leq n$ and $1<j\leq a_{i}$,
	$f_{i,j}=x_{i,1}-x_{i,j}$ forms a regular sequence of degree one in $S/\mathcal{P}(I)$. Moreover, $R/I\cong S/(\mathcal{P}(I)+J)$ where $J=(f_{i,j}\mid1\leq i\leq n, 1< j\leq a_{i})$.
\end{lemma}
For any quotient ring, factoring out with a regular sequence of degree one does not change the Betti numbers. That is, the graded Betti numbers of $R/I$ and $S/\mathcal{P}(I)$ are the same (see \cite{Froberg1982}). Hence $R/I$ has a linear resolution if and only if $S/\mathcal{P}(I)$ has a linear resolution. We show that it is also valid for componentwise linear ideals. 

Herzog and Hibi in \cite{HerzogHibi1999} defined the concept of (squarefree) componentwise linear. 
A graded ideal $I\subseteq R$ is componentwise linear if $I_{<j>}$ has a linear resolution for all $j$ where $I_{<j>}$ is the ideal generated by all homogeneous polynomials of degree $j$ belonging to $I$. An ideal $I\subseteq R$ is squarefree componentwise linear if $I_{[j]}$ has a linear resolution for all $j$ where $I_{[j]}$ is the ideal generated by all squarefree homogeneous polynomials of degree $j$ belonging to $I$. The following Theorem is standard but we include for the sake of self-containment.

%
%

\begin{theorem}\label{linear reg}
	Let $R/I$ be a graded algebra with all minimal generators of $I$ of degree $\ge k$. Then $R/I$ has a $k$-linear resolution if and only if $\reg(I)=k$.
\end{theorem}
\begin{proof}
	Suppose $R/I$ has a $k$-linear resolution
	$$0\longleftarrow R/I\longleftarrow R\longleftarrow R^{\beta_1}[-k]\longleftarrow R^{\beta_2}[-k-1]\longleftarrow\cdots\longleftarrow R^{\beta_m}[-k-m+1]\longleftarrow
	0$$
	Thus reg$(R/I)=\max\{ k+i-i-1\}=k-1$, so reg$(I)=k$ . If all generators of $I$ are of degree $\ge k$ a resolution of $R/I$ 
	looks like this:
	$$0\longleftarrow R/I\longleftarrow R\longleftarrow \oplus_{j=1}^{j_1} R^{\beta_{1,k-1+j}}[-k+1-j]\longleftarrow \oplus_{j=1}^{j_2} R^{\beta_{2,k+j}}[-k-j]\longleftarrow\cdots$$
	$$\cdots\longleftarrow\oplus_{j=1}^{j_m} R^{\beta_{m,k+m-2+j}}[-k-m+2-j]\longleftarrow0$$
	If reg$(I)=k$ (so reg$(R/I)=k-1$) then $\{ k+i-2+j_i-i\}\le k-1$ for all $i$, so $j_i=1$  for all $i$, so 
	the resolution is linear.
\end{proof}

\begin{proposition}\label{componentwise linear}
	Let $I$ be monomial ideal $R$. Then the following conditions are equivalent:
	\begin{enumerate}[\rmfamily i)]
		\item $I$ is componentwise linear;
		\item $\mathcal{P}(I)$ is componentwise linear;
		\item $\mathcal{P}(I)$ is squarefree componentwise linear.
	\end{enumerate}
\end{proposition}
\begin{proof}
	$(i)\Leftrightarrow (ii)$ Let $k$ be a positive integer. By Lemma \ref{Froberg}, $\reg(I)=\reg(\mathcal{P}(I))$ and $I$ has a linear resolution if and only if $\mathcal{P}(I)$ has a  linear resolution. Then it follows from Theorem \ref{linear reg} that
	$\mathcal{P}(I_{<k>})$ has a linear resolution if and only if $\reg(I_{<k>})=k$. But Theorem \ref{linear reg} also implies that $\mathcal{P}(I)_{<k>}$ has a linear resolution if and only if $\reg(\mathcal{P}(I)_{<k>})=k$. Hence $\mathcal{P}(I_{<k>})$ has a linear resolution if and only if $\mathcal{P}(I)_{<k>}$ has a linear resolution. Then the proof follows from the definition of componentwise linear.
	
	$(ii)\Leftrightarrow (iii)$ See {\cite[Proposition 1.5]{HerzogHibi1999}}.
\end{proof}

\begin{corollary}\label{I_k, I_geq k}
		Let $I$ be a monomial ideal of $R$. Then the following conditions are equivalent:
		\begin{enumerate}[\rmfamily i)]
			\item $I_{\geq k}$ has a linear resolution;
			\item $\mathcal{P}(I)_{\geq k}$ has a linear resolution;
			\item $\mathcal{P}(I)_k$ has a linear resolution.
		\end{enumerate}
\end{corollary}

\begin{proof}
	Note that by the argument in the proof of Proposition  \ref{componentwise linear} we have $\mathcal{P}(I_{\geq k})$ has a linear resolution if and only if $\mathcal{P}(I)_{\geq k}$ has a linear resolution. Then the proof is an immediate consequence of the statement of Proposition \ref{componentwise linear}.
\end{proof}

\begin{example}
	Let $J=\left(x_1^3,x_2^4,x_1^2x_2^2x_3^2\right)$ be an ideal of $\K[x_1,x_2,x_3]$.
	The polarization of $J$ can be thought as the ideal $I$ in Example \ref{example 1}.
	It follows from Corollaries \ref{reg linear} and \ref{I_k, I_geq k} that $I_{\geq k}$ has a linear resolution for all $k\geq7$ but it does not have for $k\leq6$.
\end{example}

\section{The graded Betti numbers of $I_{\geq k}$}
In this section we consider quotients of $I_{\ge k}=I\cap M^k$, where $M$ is the graded maximal ideal, for any graded, not necessarily monomial, ideal.
We show that we can determine all graded Betti numbers of $R/I_{\ge k}$ if we know the Betti numbers of $R/I$ and the Hilbert series of $R/I_{\ge k}$.
The following theorem is probably well known, but we haven't found any reference.

\begin{theorem}\label{betti I_geq k} 
	Let $I$ be a graded ideal in $R=k[x_1,\ldots,x_n]$. Let $k=\min\{\deg(u)\mid u\in I\}$. If $j<i+k-1$, then $\beta_{i,j}(R/I_{\ge k})=0$. If $j>i+k-1$,
	then $\beta_{i,j}(R/I_{\ge k})=\beta_{i,j}(R/I)$. If we know the graded Betti numbers of $R/I$ and the Hilbert series of $R/I_{\ge k}$, we can
	determine the graded Betti numbers of $R/I_{\ge k}$.
\end{theorem}

\begin{proof}
	We think about the Betti numbers as the dimension of the homology of the Koszul complex of $R/I$. A ``monomial" $mT_{l_1}\wedge\cdots\wedge T_{l_i}$
	is of degree $(i,i+j)$ if $m\in R/I$ is of degree $j$, and elements of degree $(i,i+j)$ are linear combinations of such monomials. Now $\beta_{1,j}(R/I_{\ge k})=0$
	if $j<k$ since the generators of $I_{\ge k}$ have degree $\ge k$, and then $\beta_{i,j}(R/I_{\ge k})=0$ if $j<i+k-1$, since if $j_i$ is the smallest $j$
	for which $\beta_{i,j}(R/I_{\ge k})\ne0$, then $j_i<j_{i+1}$ because we have a minimal resolution. In degrees $j\ge i-k-1$ $R/I_{\ge k}$ is identical 
	with $R/I$, so if $j>i+k-1$
	$\beta_{i,j}(R/I_{\ge k})=\beta_{i,j}(R/I)$. Since the Hilbert series of $R/I_{\ge k}$ equals $\sum_{i=0}^n(-1)^i\beta_{i,i+j}(R/I_{\ge k})t^{i+j}/(1-t)^n$, we can determine 
	the remaining Betti numbers $\beta_{i,i+k-1}(R/I_{\ge k})$ from the Hilbert series.
\end{proof}

We first give some corollaries.

\begin{corollary}
	We have that $k$-index of $R/I_{\ge k}$ is smaller or equal to the $(k+1)$-index. In particular, if $R/I_{\geq k}$ has a linear resolution, 
	then $R/I_{\geq k+1}$ has a linear resolution.
\end{corollary}

\begin{corollary}
	Let $I$ be a graded ideal of $R$ and $\reg(I)=d$. Then for any positive integer $k$ we have
	$$\reg(I_{\geq k})=\begin{cases}
		d\quad\mbox{ if }d\geq k,\\
		k\quad\mbox{ otherwise}.
	\end{cases}$$
\end{corollary}
\begin{proof}
	We have $\beta_{i,i+d}(I)\neq0$ for some $i$. Theorem \ref{betti I_geq k} implies that $\reg(I_{\geq k})=d$ whenever $d\geq k$ and $\reg(I_{\geq k})=k$ otherwise.
\end{proof}

For a complete intersection generated in degree 2, the result is easy to describe.

\begin{corollary}\label{ind}
	If $I$ is an artinian complete intersection generated in degree 2, then the $k$-index of $R/I_{\geq k}$ is $k-1$ if $2\le k\le n$ and the $k$-index is $\infty$ if $k\ge n+1$.
\end{corollary}

\begin{corollary}
	If $S=k[x_1,\ldots,x_n]/((f_1,\ldots,f_n)\cap M^k)$, $f_1,\ldots,f_n$ a complete intersection in degree $d$, then $S$ has a linear resolution exactly for $k\ge nd-n+1$.
\end{corollary}

\begin{corollary}
	If $S=k[x_1,\ldots,x_n]/(f_1,\ldots,f_n)$, where $f_1,\ldots,f_n$ is a complete intersection in degree $d$, and $T=k[x_1,\ldots,x_n]/((f_1,\ldots,f_n)\cap M^k)$.
	Suppose that $k\le n(d-1)$. Then $\beta_{i,j}(T)\ne0$ only if $(i,j)=(0,0)$, $(i,j)=(i,k+i-1)$ for $1\le i\le n$, and for $(i,j)=(r,dr)$ for $\frac{k-1}{d-1}<r\le n$. Furthermore,
	$\beta_{r,dr}={n\choose r}$ for those $r$. 
\end{corollary}

\begin{example}
	Let $S=k[x_1,\ldots,x_8]/\left((x_1^2,\ldots,x_8^2)\cap M^5\right)$. In $S$ we have all monomials
	of degree $<5$, but only the squarefree in degrees 6,7, and 8, and nothing in degrees $>8$.
	Thus the Hilbert series is 
	$$1+{8\choose1}t+{9\choose2}t^2+{10\choose3}t^3+{11\choose4}t^4+{8\choose3}t^5+{8\choose2}t^6+{8\choose1}t^7+{8\choose0}t^8=$$
	$$\frac{(1-t)^8(1+{8\choose1}t+{9\choose2}t^2+{10\choose3}t^3+{11\choose4}t^4+{8\choose3}t^5+{8\choose2}t^6+{8\choose1}t^7+{8\choose0}t^8)}{(1-t)^8}=$$
	$$\frac{1-736t^5+4200t^6-10528t^7+14910t^8-12832t^9+6720t^{10}-2016t^{11}+288t^{12}-8t^{14}+t^{16}}{(1-t)^8}.$$
	We have $\beta_{5,10}=56,\beta_{6,12}=28,\beta_{7,14}=8,\beta_{8,16}=1$. The Hilbert series gives $\beta_{0,0}=1,\beta_{1,5}=736,\beta_{2,6}=4200,\beta_{3,7}=10528,\beta_{4,8}=14910,\beta_{5,9}=12832,\beta_{6,10}=6720+56=6776,\beta_{7,11}=2016,
	\beta_{8,12}=288-28=260$.
\end{example}

\bibliographystyle{plain}
\bibliography{Truncation}

\end{document}